\documentclass{amsart2020}

\newtheorem{theorem}{Theorem}[section]
\newtheorem{lemma}[theorem]{Lemma}

\newtheorem{corollary}[theorem]{Corollary}

\newtheorem{problem}[theorem]{Problem}

\theoremstyle{definition}

\newtheorem{example}[theorem]{Example}
\newtheorem{remark}[theorem]{Remark}

\usepackage{amscd,amssymb}

\begin{document}

\title[Minimal varieties and transcendental series]
{Minimal varieties of associative algebras\\
and transcendental series}

\author[Vesselin Drensky]
{Vesselin Drensky}

\address{Institute of Mathematics and Informatics,
Bulgarian Academy of Sciences,
1113 Sofia, Bulgaria}
\email{drensky@math.bas.bg}

\thanks{Partially supported by
Grant KP-06 N 32/1 of 07.12.2019 ``Groups and Rings -- Theory and Applications'' of the Bulgarian National Science Fund.}


\maketitle

\centerline{\it To Antonio Giambruno for his anniversary}

\begin{abstract}
A variety of associative algebras over a field of characteristic 0 is called minimal
if its codimension sequence grows much faster than the codimension sequence of any of its proper subvarieties.
By the results of Giambruno and Zaicev it follows that the number $b_n$ of minimal varieties of given exponent $n$ is finite.
Using methods of the theory of colored (or weighted) compositions of integers,
we show that the limit $\beta=\lim_{n\to\infty}\sqrt[n]{b_n}$ exists and can be expressed as the positive solution of
an equation $a(t)=0$ where $a(t)$ is an explicitly given power series.
Similar results are obtained for the number of minimal varieties with a given Gelfand-Kirillov dimension of their relatively free algebras of rank $d$.
It follows from classical results on lacunary power series that the generating function of the sequence $b_n$, $n=1,2,\ldots$, is transcendental.
With the same approach we construct examples of free graded semigroups $\langle Y\rangle$ with the following property.
If $d_n$ is the number of elements of degree $n$ of $\langle Y\rangle$,
then the limit $\delta=\lim_{n\to\infty}\sqrt[n]{d_n}$ exists and is transcendental.
\end{abstract}

\section*{Introduction}

Let $K\langle X\rangle=K\langle x_1,x_2,\ldots\rangle$ be the free associative algebra of countable rank over a field $K$.
The polynomial $f(x_1,\ldots,x_n)\in K\langle X\rangle$ is a polynomial identity for the algebra $R$ if
\[
f(r_1,\ldots,r_n)=0\text{ for all }r_1,\ldots,r_n\in R.
\]
The algebra $R$ is called a PI-algebra if it satisfies a nontrivial polynomial identity.
The class $\mathfrak R$ of all algebras satisfying a given system of identities is the variety of algebras defined by the system.
The sets $T(R)$ and $T({\mathfrak R})$ of all polynomial identities of the algebra $R$ and the variety $\mathfrak R$, respectively,
are invariant under the action of all endomorphisms of $K\langle X\rangle$ and are called T-ideals.
If $T({\mathfrak R})=T(R)$, then the variety $\mathfrak R$ is generated by the algebra $R$.
It is well known that every PI-algebra satisfies a multilinear polynomial identity of some degree $n$,
i.e., an identity from the vector space
\[
P_n=\text{span}\{x_{\sigma(1)}\cdots x_{\sigma(n)}\mid \sigma\in S_n\},
\]
where $S_n$ is the symmetric group of degree $n$.
Regev \cite{Re1} proved that the codimension sequence of a proper subvariety $\mathfrak R$ of the variety of all associative algebras
\[
c_n({\mathfrak R})=\dim(P_n/(P_n\cup T(\mathfrak R)),\quad n=1,2,\ldots,
\]
(and the codimension sequence $c_n(R)$ for a PI-algebra $R$)
is exponentially bounded, i.e., there exists a constant $a$ such that $c_n({\mathfrak R})\leq a^n$, $n=1,2,\ldots$.

When the base field $K$ is of characteristic 0 the identities of the algebra are determined by the multilinear ones.
Later we shall consider PI-algebras over a field $K$ of characteristic 0 only.
Amitsur conjectured that the limit
\[
\exp({\mathfrak R})=\lim_{n\to\infty}\sqrt[n]{c_n({\mathfrak R})}
\]
exists and is a nonnegative integer.
This conjecture was confirmed by Giambruno and Zaicev \cite{GZ1, GZ2}.

Kemer \cite{K1} developed the structure theory of T-ideals which was in the ground of his solution of the Specht problem
for the existence of a finite basis of the identities of any variety of associative algebra, see his book \cite{K2} for an account.
The key role of the theory of Kemer is played by the T-prime ideals
\[
T(M_m(K)),T(M_m(E)), T(M_{a,b}),\quad m\geq 1,1\leq a\leq b,
\]
where $M_m(K)$ and $M_m(E)$ are, respectively, the $m\times m$ matrix algebras with entries from the field $K$
and from the infinite dimensional Grassmann algebra $E$.
The algebra $M_{a,b}$ consists of block matrices
$\displaystyle \left(\begin{matrix}A_{11}&A_{12}\\
A_{21}&A_{22}\\
\end{matrix}\right)$, where $A_{11}$ and $A_{22}$ are, respectively, $a\times a$ and $b\times b$ matrices
with entries from the even component $E_0$ of $E$,
and $A_{12}$ and $A_{21}$ are, respectively, $a\times b$ and $b\times a$ matrices
with entries from the odd component $E_1$ of $E$.

The exponent $\exp({\mathfrak R})$ of the variety of algebras $\mathfrak R$ is one of the natural ways to measure
the complexity of the polynomial identities of $\mathfrak R$. The variety $\mathfrak R$ is minimal
if $\exp({\mathfrak R})>\exp({\mathfrak S})$ for all proper subvarieties $\mathfrak S$ of $\mathfrak R$.
Drensky \cite{D1, D2} conjectured that the varieties with T-ideals which are products of T-prime ideals are minimal.
This was confirmed by Giambruno and Zaicev \cite{GZ3, GZ4} who showed that a variety $\mathfrak R$ is minimal if and only if
its T-ideal $T({\mathfrak R})$ is a product of T-prime ideals.
The paper \cite{GZ3} considers the case of varieties of finite basic rank when the T-ideal of the minimal variety
is a product of T-ideals of matrix algebras with entries from the field.
The case of infinite basic rank \cite{GZ4} involves also some of the T-ideals $T(M_m(E))$ and $T(M_{a,b})$ in the product.
We refer to the book by Giambruno and Zaicev \cite{GZ5}
for an account of the asymptotical results and methods in the theory of PI-algebras.

It follows from the description of minimal varieties that there is a finite number of minimal varieties of a given exponent.
It is a natural question to ask how many are they.
Again, it is convenient to consider separately the cases of minimal varieties of finite basic rank
and arbitrary minimal varieties. Let $b_n^{(\text{\rm f.g. codim})}$ and $b_n^{(\text{\rm codim})}$ be, respectively,
the number of minimal varieties of exponent $n$ of finite and arbitrary basic rank.
Using methods of the theory of colored (or weighted) compositions of integers, it is easy to write the generating functions
\[
b^{(\text{\rm f.g. codim})}(t)=\sum_{n\geq 0}b_n^{(\text{\rm f.g. codim})}t^n\text{ and }
b^{(\text{\rm codim})}(t)=\sum_{n\geq 0}b_n^{(\text{\rm codim})}t^n.
\]
It seems that it is a hopeless problem to express explicitly the integers $b_n^{(\text{\rm f.g. codim})}$ and $b_n^{(\text{\rm codim})}$.
As an approximation to the problem we show that both limits
\[
\exp(\text{\rm f.g. codim})=\lim_{n\to\infty}\sqrt[n]{b_n^{(\text{\rm f.g. codim})}}\text{ and }
\exp(\text{\rm codim})=\lim_{n\to\infty}\sqrt[n]{b_n^{(\text{\rm codim})}}.
\]
exist. It follows from classical results on lacunary power series that the functions $b^{(\text{\rm f.g. codim})}(t)$
and $b^{(\text{\rm codim})}(t)$ transcendental. We believe that $\exp(\text{\rm f.g. codim})$ and $\exp(\text{\rm codim})$
are transcendental numbers but cannot prove it.

The most convenient way to measure the growth of a  finitely generated algebra $R$ is to use its Gelfand-Kirillov dimension $\text{GKdim}(R)$.
For relatively free algebras
\[
F_d({\mathfrak R})=K\langle x_1,\ldots,x_d\rangle/(K\langle x_1,\ldots,x_d\rangle\cap T({\mathfrak R}))
\]
the Gelfand-Kirillov dimension provides also information about the quantity of
the polynomial identities of $\mathfrak R$ in $d$ variables.
A combination of the results of Kemer \cite{K3} and Markov \cite{M} gives that for any proper subvariety $\mathfrak R$
of the variety of all associative algebras $\text{GKdim}(F_d({\mathfrak R}))$ is an integer.

Using our approach for the asymptotics of the number of minimal varieties of a given exponent, we obtain analogous results for
the number $b_n^{(d\text{\rm -f.g. generated})}$ of minimal varieties $\mathfrak R$ of finite basic rank such that for a fixed integer $d>1$
the Gelfand-Kirillov dimension of the $d$-generated relatively free algebra $\text{GKdim}(F_d({\mathfrak R}))$ is equal to $n$,
and the similar number $b_n^{(d\text{\rm -generated})}$ for all minimal varieties.
Again, the generating function of the sequence $b_n^{(d\text{\rm -f.g. generated})}$, $n=1,2,\ldots$, is transcendental
but we do not know whether the same holds for the sequence $b_n^{(d\text{\rm -generated})}$, $n=1,2,\ldots$.

With the same approach we construct examples of free graded semigroups $\langle Y\rangle$ with the following property.
If $g_n$ is the number of elements of degree $n$ of $\langle Y\rangle$,
then the limit $\displaystyle \gamma=\lim_{n\to\infty}\sqrt[n]{g_n}$ exists and is transcendental.
As in the case of minimal varieties in our examples the generating function of the sequence $g_n$, $n=0,1,2,\ldots$, is also transcendental.
Clearly, the same properties has the Hilbert series $H(K\langle Y\rangle,t)$ and its coefficients
of the free graded algebra $K\langle Y\rangle$ freely generated by the set $Y$ over a field $K$ of arbitrary characteristic.

The paper is organized as follows. In Section \ref{Section Colored compositions} we summarize the necessary facts on colored compositions,
Section \ref{Section Transcendental series} is devoted to the properties of lacunary power series.
The main results on the number of minimal varieties with given property are in Section \ref{Section Minimal varieties}.
Finally we give the examples of free graded semigroups and free associative algebras
discussed above.

\section{Colored compositions}\label{Section Colored compositions}

In this section we shall collect facts, most of them well known, which we shall need for the study of the number of minimal varieties of given exponent
and related problems. For a background on theory of compositions we refer to the book by Heubach and Mansour \cite{HeM}.

A composition $\nu=(\nu_1,\ldots,\nu_m)$ of a positive integer $n$
is a representation of $n$ as a sum of positive integers $\nu_1,\ldots,\nu_m$
taking into account the order of the summands. We shall be interested in colored (or weighted) $A$-compositions.
This means that the summands are from a given nonempty subset $A$ of $\mathbb N$ and the elements $\nu_i$ of $A$
participate in $A$ with some weights (or multiplicities) $a_{\nu_i}$, and the equal $\nu_i$'s are in different colors
in order to be distinguished.

For example, if $A=\{1,\text{\bf 2},\text{\it 2},3,5,6\}$, then the compositions
$(\text{\bf 2},1,\text{\it 2},\text{\bf 2},5)$ and $(\text{\it 2},1,\text{\bf 2},\text{\it 2},5)$ are different.

In the sequel we shall fix the notation
\[
a(t)=\sum_{k\geq 1}a_kt^k\text{ and } b(t)=\sum_{n\geq 0}b_nt^n
\]
for, respectively, the generating function of $A$ counting the multiplicities of its elements
and the generating function of the number of colored $A$-compositions of $n$.
If the elements of $A$ have a common divisor $d>1$, then all nonzero coefficients $b_n$ of the power series $b(t)$
have subscripts $n$ divisible by $d$ only.

If the elements of $A$ are coprime, then all coefficients $b_n$ are nonzero
for $n$ large enough. It is well known that if $r$ is the radius of convergence of the power series $b(t)$, then
\[
b_n\bowtie \frac{1}{r^n},\text{ i.e. }b_n=\frac{\vartheta(n)}{r^n},\quad \limsup_{n\to\infty}\sqrt[n]{\vartheta(n)}=1.
\]
It is a natural question what is the behavior of the function $\vartheta(n)$. In particular, is it true that
\[
\limsup_{n\to\infty}\sqrt[n]{b_n}=\liminf_{n\to\infty}\sqrt[n]{b_n}\text{ or, equivalently, whether }
\lim_{n\to\infty}\sqrt[n]{b_n}\text{ exists?}
\]
We shall see that $\displaystyle \lim_{n\to\infty}\sqrt[n]{b_n}$ does exist and shall show how to express it in terms of $a(t)$.

The following statement is well known, see, e.g. \cite[Section 3.5]{HeM} or Moser and Whitney \cite{MW}.

\begin{lemma}\label{generating function for colored compositions}
The generating function $b(t)$ which counts the number of colored $A$-compositions of $n$ is expressed
in terms of the generating function $a(t)$ of the multiset $A$ as
\[
b(t)=\frac{1}{1-a(t)}.
\]
\end{lemma}

In the applications in Section \ref{Section Minimal varieties} the radius of convergence of the power series $a(t)$ is equal to 1.
We shall need the following lemma.

\begin{lemma}\label{Abel's theorem}
If the radius of convergence of the power series $\displaystyle a(t)=\sum_{k\geq 1}a_kt^k$ with nonnegative integer coefficients $a_k$
is equal to $1$, then the equation $a(t)=1$ has a unique positive solution.
\end{lemma}
\begin{proof}
By the theorem of Abel for power series, if the series $\displaystyle a(t)=\sum_{k\geq 1}a_kt^k$ has a radius of convergence equal to 1
and the series $\displaystyle \sum_{k\geq 1}a_k$ converges, then
\[
\lim_{t\to 1^{-}}a(t)=\sum_{k\geq 1}a_k.
\]
The same holds when the series $\displaystyle \sum_{k\geq 1}a_k$ diverges to infinity. In our case this gives that
the function $a(t)$ is continuous in the interval $[0,1)$, strictly monotonically increasing
and hence takes exactly once all values from $a(0)=0$ to infinity.
In particular, the equation $a(t)=1$ has a unique solution in the interval $(0,1)$.
\end{proof}

\begin{remark}\label{radius of convergence essential condition}
The following example shows that the condition
that the radius of convergence of the power series with nonnegative integer coefficients is equal to 1
is essential in Lemma \ref{Abel's theorem}. We start with the Riemann $\zeta$-function
\[
\zeta(s)=\sum_{k\geq 1}\frac{1}{k^s}
\]
and the well known equality
\[
\zeta(2n)=\frac{(-1)^{n+1}B_{2n}(2\pi)^{2n}}{2(2n)!},\quad n\geq 1,
\]
where $B_{2n}$ is the $2n$-th Bernoulli number. It is also well known that $\zeta(2n)<2$, $n\geq 1$.
For example,
\[
\zeta(2)=1+\frac{1}{2^2}+\frac{1}{3^2}+\cdots =\frac{\pi^2}{6}=1.6449\ldots,
\]
\[
\zeta(4)=1+\frac {1}{2^4}+\frac {1}{3^4}+\cdots =\frac{\pi^4}{90}=1.0823\ldots.
\]
Now we consider the power series
\[
a^{(n)}(t)=\sum_{k\geq 2}\left\lceil\frac{2^k}{k^{2n}}\right\rceil t^k,\quad n\geq 1,
\]
where $\lceil m\rceil$ is the integer part of $m$. Obviously the radius of convergence of $a^{(n)}(t)$
is equal to $\displaystyle\frac{1}{2}$ and
\[
a^{(n)}\left(\frac{1}{2}\right)<\zeta(2n)-1<1.
\]
Hence the equation $a^{(n)}(t)=1$ does not have a positive solution.
\end{remark}

The following easy theorem is the main tool in our further considerations.
We provide a detailed proof for completeness of the exposition.

\begin{theorem}\label{existence of exponent}
Let the multiset $A$ with generating function $a(t)$ have the property
that the greatest common divisor of its elements is equal to $1$.
Let the series $a(t)$ have a radius of convergence $\varrho$ equal to $1$
or $\varrho<1$ and $a(\varrho)\geq 1$.
Then for the coefficients of the generating function $b(t)$ of the $A$-compositions the limit
\[
\beta_A=\lim_{n\to\infty}\sqrt[n]{b_n}
\]
exists and $\beta^{-1}_A$ is equal to the positive solution $\alpha$ of the equation $a(t)=1$.
\end{theorem}

\begin{proof}
The set $S$ of all integers $n$ which allow $A$-compositions forms an additive subsemigroup of ${\mathbb N}_0$
generated by the elements of $A$. By a well known result on numerical semigroups, see, e.g. \cite{RG-S},
$S$ is a numerical semigroup since the greatest common divisor of the elements in $A$ is equal to 1.
Hence there exists an $n_0$ such that all integers $n\geq n_0$ belong to $S$.
The radius of convergence of the power series $\displaystyle \frac{1}{1-z}$ is equal to 1.
Since $\displaystyle b(t)=\frac{1}{1-a(t)}$ by Lemma \ref{generating function for colored compositions},
the radius of convergence of the series $b(t)$
is equal to the unique positive solution $\alpha$ of the equation $a(t)=1$.
As we mentioned in the beginning of the section,
\[
b_n=\beta^n\vartheta(n),\quad
\beta=\frac{1}{\alpha},\quad \limsup_{n\to\infty}\sqrt[n]{\vartheta(n)}=1.
\]
Hence $b_n<(\beta+\varepsilon)^n$ for any $\varepsilon>0$ and $n$ sufficiently large.
Let $0<\varepsilon<\beta$. There is a sequence $n_1<n_2<\cdots$ such that
\[
b_{n_i}>(\beta-\varepsilon)^{n_i},\quad i=1,2,\ldots.
\]
From two $A$-compositions $\nu=(\nu_1,\ldots,\nu_m)$ and $\nu'=(\nu'_1,\ldots,\nu'_{m'})$ of $n_1$
we obtain an $A$-composition $(\nu,\nu')=(\nu_1,\ldots,\nu_m,\nu'_1,\ldots,\nu'_{m'})$ of $2n_1$.
Hence $b_{2n_1}\geq b_{n_1}^2$ and, similarly,
\[
b_{pn_1}\geq b_{n_1}^p>(\beta-\varepsilon)^{pn_1},\quad \sqrt[pn_1]{b_{pn_1}}>\beta-\varepsilon,\quad p=1,2,\ldots.
\]
Since $b_{n_0},b_{n_0+1},\cdots,b_{2n_0}\geq 1$, with the same arguments we obtain that
\[
b_{n_1+n_0},b_{n_1+n_0+1},\ldots,b_{n_1+2n_0}\geq b_{n_1},
\]
\[
b_{(n_1+n_0)+n_0},b_{(n_1+n_0)+n_0+1},\ldots,b_{(n_1+n_0)+2n_0}=b_{n_1+3n_0}\geq b_{n_1},
\]
\[
b_{n_1+q}\geq b_{n_1}>(\beta-\varepsilon)^{n_1},\quad q\geq n_0.
\]
Hence
\[
b_{2n_1+q}\geq b_{n_1}>(\beta-\varepsilon)^{n_1},\quad b_{pn_1+q}\geq b_{(p-1)n_1}>(\beta-\varepsilon)^{(p-1)n_1}\quad p\geq 2, q>0,
\]
and for $p\geq 2$, $0\leq q<n_1$, $n=pn_1+q>2n_1$,
\[
b_n>(\beta-\varepsilon)^{(p-1)n_1}=\frac{(\beta-\varepsilon)^n}{(\beta-\varepsilon)^{n_1+q}},
\quad \sqrt[n]{b_n}>(\beta-\varepsilon)\sqrt[n]{\frac{1}{(\beta-\varepsilon)^{n_1+q}}}.
\]
Since
\[
\lim_{n\to\infty}\sqrt[n]{\frac{1}{(\beta-\varepsilon)^{n_1+q}}}=1,
\]
we obtain that
\[
\lim_{n\to\infty}\sqrt[n]{b_n}=\beta
\]
and this completes the proof.
\end{proof}

\begin{remark}
If, as in Remark \ref{radius of convergence essential condition}, the radius of convergence $\varrho$ of the power series $a(t)$
with nonnegative integer coefficients is less than 1 and $a(\varrho)<1$
(or the series diverges to infinity for any $t>0$), then Theorem \ref{existence of exponent} is still valid if we replace
the positive solution of the equation $a(t)=1$ with $\varrho$ or with $0$ in the case of divergence.
\end{remark}

\begin{remark}
If the multiset $A$ with generating function $a(t)$ has the property
that the greatest common divisor of its elements is equal to $d$, then we can rewrite $a(t)$ and $b(t)$ in the form
\[
a(t)=\sum_{k\geq 1}a_{dk}t^{dk},\quad b(t)=\sum_{n\geq 0}b_{dn}t^{dn}
\]
and, as in Theorem \ref{existence of exponent}, we obtain the existence of the limit
\[
\beta_A=\lim_{n\to\infty}\sqrt[dn]{b_{dn}}.
\]
\end{remark}

\section{Lacunary series and transcendentance}\label{Section Transcendental series}

Recall that the power series
\[
a(t)=\sum_{k\geq 1}a_kt^k
\]
is lacunary if there exist two sequences $p_1<p_2<\cdots$ and $q_1<q_2<\cdots$ of positive integers such that
\[
p_1\leq q_1<p_2\leq q_2<\cdots,
\]
$a_k=0$ for $k=q_i+1,q_i+2,\ldots,p_{i+1}-1$, $i=1,2,\ldots$, and $\displaystyle \lim_{i\to\infty}(p_{i+1}-q_i)=\infty$.
The lacunary series $a(t)$ is strongly lacunary if, in the above notation, $\displaystyle \lim_{i\to\infty}\frac{p_{i+1}}{q_i}=\infty$.

The following theorem can be found in Mahler \cite[p. 42]{Ma2}.

\begin{theorem}\label{transcendental lacunary series}
Every lacunary series $a(t)$ with integer coefficients is transcendental over ${\mathbb Q}(t)$.
\end{theorem}

The transcendence of the series $a(t)$ does not guarantee that the positive solution of the equation $a(t)=1$
is a transcendental number. We shall use the following partial case of a theorem of Cohn \cite{C}.

\begin{theorem}\label{evaluations of lacunary series}
Let
\[
a(t)=\sum_{i\geq 1}a_{k_i}t^{k_i},\quad k_i\in{\mathbb N},\quad 0< k_1<k_2<\cdots,
\]
be a strongly lacunary series such that
\[
\lim_{i\to\infty}\frac{k_{i+1}}{\log(\max\{a_{k_1},\ldots, a_{k_i}\})}=\infty.
\]
Then $a(\alpha)$ is transcendental for every nonzero algebraic $\alpha$ within the circle of convergence of $a(t)$.
\end{theorem}

Mahler \cite{Ma1} removed the condition for the limit for $\log(\max\{a_{k_1},\ldots, a_{k_i}\})$
and proved a version of Theorem \ref{evaluations of lacunary series} for strongly lacunary series.

Maybe the most famous strongly lacunary series is
\[
a(t)=\sum_{k\geq 1}t^{k!}.
\]
Its evaluation
\[
a\left(\frac{1}{10}\right)=\sum_{k\geq 1}\frac{1}{10^{k!}}
\]
is equal to the first explicitly constructed transcendental number found by Liouville \cite{L}.
For a survey on the recent results on transcendence of values of power series see, e.g. Kaneko \cite{Ka}.
As immediate consequences of Theorem \ref{evaluations of lacunary series} and Lemma \ref{Abel's theorem} we obtain the following.

\begin{corollary}\label{transcendental radius of convergence}
Let $a(t)$ be a power series with nonnegative integer coefficients which satisfies the conditions of Theorem \ref{evaluations of lacunary series}
and has a radius of convergence equal to $1$.
Then the positive solution of the equation $a(t)=1$ is a transcendental number.
\end{corollary}

\begin{corollary}\label{transcendental radius of convergence for A-compositions}
Let $a(t)$ be a power series with nonnegative integer coefficients which satisfies the conditions of Corollary \ref{transcendental radius of convergence}
and has a radius of convergence equal to $1$. Let $A$ be the associated to $a(t)$ multiset. Then for the number $b_n$ of the $A$-compositions of $n$
\[
\beta_A=\lim_{n\to\infty}\sqrt[n]{b_n}
\]
is a transcendental number.
\end{corollary}

\begin{example}\label{Liouville number}
If we take $A=\{1!,2!,3!,\ldots\}$, then Corollary \ref{transcendental radius of convergence for A-compositions}
gives that
\[
\beta_A=\lim_{n\to\infty}\sqrt[n]{b_n}
\]
is a transcendental number.
\end{example}

\section{Minimal varieties of algebras}\label{Section Minimal varieties}

This section contains the main results of the project.
In what follows $\mathfrak R$ is a minimal variety of associative algebras over a field $K$ of characteristic 0.
This means that the T-ideal $T({\mathfrak R})$ is a product of the T-prime ideals
$T(M_m(K))$, $T(M_m(E))$, $m\geq 1$, $T(M_{a,b})$, $1\leq a\leq b$. The variety is of finite basic rank if
$T({\mathfrak R})$ is a product of T-ideals $T(M_{m_i}(K))$ only.

The following statement is proved in the book by Giambruno and Zaicev \cite[Corollary 8.5.5, p. 210]{GZ5}.
(As the authors mention in \cite{GZ5}, the freeness of this semigroup
can be deduced also from the result of Bergman and Lewin \cite[Theorem 7]{BeL}).

\begin{lemma}\label{free graded semigroup of minimal varieties}
The T-ideals of the minimal varieties
form a free graded semigroup freely generated by the T-prime ideals. The grading is defined by
\[
\deg T({\mathfrak R})=\exp({\mathfrak R}).
\]
\end{lemma}

The asymptotics of the codimension sequence of the T-prime ideals was found by Regev \cite{Re2} for
$T(M_m(K))$ and by Berele and Regev \cite{BR} for $T(M_m(E))$ and $T(M_{a,b})$.
It follows from there that
\[
\exp(M_m(K))=m^2,\quad \exp(M_m(E))=2m^2,\quad \exp(M_{a,b})=(a+b)^2.
\]
We shall consider two kinds of generating functions -- for varieties of finite and of infinite basic rank. We fix the notation
\[
a^{(\text{\rm f.g. codim})}(t)=\sum_{m\geq 1}t^{\exp(M_m(K))}=\sum_{m\geq 1}t^{m^2}
\]
for the generating function for the exponents of the codimension sequences of T-prime varieties in the case of finite basic rank and
\[
a^{(\text{\rm codim})}(t)=\sum_{m\geq 1}t^{\exp(M_m(K))}+\sum_{m\geq 1}t^{\exp(M_m(E))}+\sum_{0<a\leq b}t^{\exp(M_{a,b})}
\]
\[
=\sum_{m\geq 1}t^{m^2}+\sum_{m\geq 1}t^{2m^2}+\sum_{1\leq a\leq b}t^{(a+b)^2}
\]
in the case of all T-prime varieties. Similarly,
\[
b^{(\text{\rm f.g. codim})}(t)=\sum_{n\geq 0}b^{(\text{\rm f.g. codim})}_nt^n,
\]
\[
b^{(\text{\rm codim})}(t)=\sum_{n\geq 0}b^{(\text{\rm codim})}_nt^n
\]
are the corresponding generating functions for the number of minimal varieties.

\begin{theorem}\label{main theorem}
Let $b^{(\text{\rm f.g. codim})}(t)$ and $b^{(\text{\rm codim})}(t)$ be, respectively, the generating functions
counting the number of minimal varieties $\mathfrak R$ of associative algebras
over a field $K$ of characteristic $0$ of a given exponent for the varieties of finite basic rank
and for all minimal varieties. Then
\[
\lim_{n\to\infty}\sqrt[n]{b_n^{(\text{\rm f.g. codim})}}\text{ and }\lim_{n\to\infty}\sqrt[n]{b_n^{(\text{\rm codim})}}
\]
exist and are equal to $\alpha^{-1}$, where $\alpha$ is, respectively, the positive solution of the equation
$a^{(\text{\rm f.g. codim})}(t)=1$ and $a^{(\text{\rm codim})}(t)=1$.
\end{theorem}

\begin{proof}
Since $\exp(M_1(K))=\exp(K)=1$,
the greatest common divisor of the indices of the nonzero coefficients of the power series $a^{(\text{\rm f.g. codim})}(t)$
and $a^{(\text{\rm codim})}(t)$ is equal to 1 and we can apply Theorem \ref{existence of exponent}.
Since the coefficients of the power series
$\displaystyle \sum_{m\geq 1}t^{m^2}$ and $\displaystyle \sum_{m\geq 1}t^{2m^2}$ are equal to 0 or 1,
their radius of convergence is equal to 1. This completes the case of minimal varieties of finite basic rank.
For the third power series $\displaystyle \sum_{1\leq a\leq b}t^{(a+b)^2}$ we have that
\[
\sum_{1\leq a\leq b}t^{(a+b)^2}\preceq\sum_{m\geq 1}(m-1)t^{m^2}\preceq \frac{d}{dt}\left(\sum_{m\geq 1}t^{m^2+1}\right)
\]
where $u(t)\preceq v(t)$ means that the coefficients of $u(t)$ are less or equal to the corresponding coefficients of $v(t)$.
This implies that the radius of convergence of $a^{(\text{\rm codim})}(t)$ is also equal to 1.
\end{proof}

\begin{theorem}\label{transcendental series for exponents}
The power series $b^{(\text{\rm f.g. codim})}(t)$ and $b^{(\text{\rm codim})}(t)$ are transcendental over ${\mathbb Q}(t)$.
\end{theorem}

\begin{proof}
It is sufficient to establish the transcendence of the series $a^{(\text{\rm f.g. codim})}(t)$ and $a^{(\text{\rm codim})}(t)$.
Obviously the series $a^{(\text{\rm f.g. codim})}(t)$ is lacunary. It is sufficient to choose $p_i=q_i=i^2$. Then
\[
\lim_{i\to\infty}(p_{i+1}-q_i)=\lim_{i\to\infty}(2i+1)=\infty
\]
and the transcendence of $a^{(\text{\rm f.g. codim})}(t)$ follows immediately from Theorem \ref{transcendental lacunary series}.
For the proof of the transcendence of $a^{(\text{\rm codim})}(t)$ we shall show that this series is also lacunary.
Let $j^2$ be between $2i^2$ and $2(i+1)^2$. Then $i\sqrt{2}<j<(i+1)\sqrt{2}$, i.e. $j$ belongs to the interval $(i\sqrt{2},(i+1)\sqrt{2})$
of length $\sqrt{2}$. Hence there is an least one and not more than two integers $j$ in this interval.
Such $j$ (or $j'<j''$) divide the interval $[2i^2,2(i+1)^2]$ of length $2(2i+1)$ in two parts $[2i^2,j^2]$, $[j^2,2(i+1)^2]$
(or in three parts $[2i^2,(j')^2]$, $[(j')^2,(j'')^2]$, $[(j'')^2,2(i+1)^2]$). Denoting the longest interval by $[q_i,p_{i+1}]$
we obtain that
\[
\lim_{i\to\infty}(p_{i+1}-q_i)\geq\lim_{i\to\infty}\frac{2(2i+1)}{3}=\infty
\]
which completes the proof.
\end{proof}

\begin{remark}
Easy computations show that $\alpha\approx 0.7054$ for the solution of the equation $a^{(\text{\rm f.g. codim})}(t)=1$
and hence $\displaystyle \lim_{n\to\infty}\sqrt[n]{b_n^{(\text{\rm f.g. codim})}}\approx 1.4176$.
On the other hand, the first values $\displaystyle \sqrt[n]{b_n^{(\text{\rm f.g. codim})}}$ for $n=1,2,\ldots,16$, are approximately equal to
\[
1,\ 1,\ 1,\ 1.1892,\ 1.2457,\ 1.2599,\ 1.2584,\ 1.2753,\ 1.3052,\ 1.3195,
\]
\[
1.3244,\ 1.3276,\ 1.3355,\ 1.3428,\ 1.3478,\ 1.3515,
\]
\[
\sqrt[24]{b_{24}^{(\text{\rm f.g. codim})}}\approx 1.3732.
\]
\end{remark}

We believe that the limits $\displaystyle \lim_{n\to\infty}\sqrt[n]{b_n^{(\text{\rm f.g. codim})}}$
and $\displaystyle \lim_{n\to\infty}\sqrt[n]{b_n^{(\text{\rm codim})}}$ are transcendental numbers
but cannot prove it.
In the case of finite basic rank this can be restated as the following problem.

\begin{problem}
Let $b_n$ be the number of $A$-compositions of $n$ for $A=\{1^2,2^2,\ldots\}$.
Is the limit $\displaystyle \lim_{n\to\infty}\sqrt[n]{b_n}$ transcendental?
\end{problem}

Recall the definition of the Gelfand-Kirillov dimension $\text{GKdim}(R)$ of a finitely generated algebra $R$.
If $R$ is generated by a finite dimensional vector space $V$, then the growth function of $R$ with respect to $V$ is
\[
g_V(n)=\dim(V^0+V^1+\cdots +V^n),\text{ where }V^j=\text{span}\{r_{i_1}\cdots r_{i_j}\mid r_{ik}\in V\},
\]
and the Gelfand-Kirillov dimension of $R$ is
\[
\text{GKdim}(R)=\limsup_{n\to\infty}(\log_n(g_V(n)))
\]
if this upper bound exists. For a background on Gelfand-Kirillov dimension we refer to the book by Krause and Lenagan \cite{KL}
and for applications to PI-algebras the books \cite{D4, GZ5} and the survey article \cite{D3}.
In particular by a theorem of Berele \cite{B1}, $\text{GKdim}(R)<\infty$ for any finitely generated PI-algebra $R$.
Kemer \cite{K3} proved that for any proper subvariety $\mathfrak R$ of the variety of all associative algebras
the relatively free algebra $F_d({\mathfrak R})$ is representable, i.e. there exists an extension $L$ of the base field $K$
and an $m\geq 1$  such that $F_d({\mathfrak R})$ is isomorphic to a subalgebra of $M_m(L)$ considered as a $K$-algebra.
Markov \cite{M} showed that the Gelfand-Kirillov dimension of finitely generated representable algebras is an integer.
This gives that for the relatively free algebra $F_d({\mathfrak R})$ the Gelfand-Kirillov dimension $\text{GKdim}(F_d({\mathfrak R}))$ is an integer.

Now we shall apply our approach for the behavior of the number of minimal varieties of a given exponent
to the number $b_n^{(d\text{\rm -f.g. generated})}$ of minimal varieties $\mathfrak R$ of finite basic rank such that for a fixed integer $d>1$
the Gelfand-Kirillov dimension of the $d$-generated relatively free algebra $\text{GKdim}(F_d({\mathfrak R}))$ is equal to $n$,
and the similar number $b_n^{(d\text{\rm -generated})}$ for all minimal varieties.

The Gelfand-Kirillov dimension of the relatively free algebras of the T-prime varieties was calculated by Procesi \cite{P}
for $\text{\rm var}(M_m(K))$ and by Berele \cite{B2} for $\text{\rm var}(M_m(E))$ and $\text{\rm var}(M_{a,b})$:
\[
\text{\rm GKdim}(F_d(\text{\rm var}(M_m(K))))=(d-1)m^2+1,
\]
\[
\text{\rm GKdim}(F_d(\text{\rm var}(M_m(E))))=(d-1)m^2+1,
\]
\[
\text{\rm GKdim}(F_d(\text{\rm var}(M_{a,b})))=(d-1)(a^2+b^2)+2,
\]
We define
\[
a^{(d\text{\rm -f.g. generated})}(t)=\sum_{m\geq 1}t^{\text{\rm GKdim}(F_d(\text{\rm var}(M_m(K))))}=\sum_{m\geq 1}t^{(d-1)m^2+1},\quad d\geq 2,
\]
for the generating function of the Gelfand-Kirillov dimension of the $d$-generated relatively free algebras
in the case of T-prime varieties of finite basic rank and
\[
a^{(d\text{\rm -generated})}(t)=2\sum_{m\geq 1}t^{(d-1)m^2+1}+\sum_{1\leq a\leq b}t^{(d-1)(a^2+b^2)+2},\quad d\geq 2,
\]
for all T-prime varieties.

\begin{theorem}\label{asymptotics for GKdim}
Let $b^{(d\text{\rm -f.g. generated})}(t)$ and $b^{(d\text{\rm -generated})}(t)$ be, respectively, the generating functions
counting the number of minimal varieties $\mathfrak R$ of associative algebras
over a field $K$ of characteristic $0$ with $d$-generated relatively free algebras of a given Gelfand-Kirillov dimension
for the varieties of finite basic rank
and for all minimal varieties. Then the limits
\[
\lim_{n\to\infty}\sqrt[n]{b_n^{(d\text{\rm -f.g. generated})}}\text{ and }\lim_{n\to\infty}\sqrt[n]{b_n^{(d\text{\rm -generated})}}
\]
for the coefficients of these generating functions exist and are equal to $\alpha^{-1}$,
where $\alpha$ is equal, respectively, to the positive solution of the equation
$a^{(d\text{\rm -f.g. generated})}(t)=1$ and $a^{(d\text{\rm -generated})}(t)=1$.
\end{theorem}

\begin{proof}
As in the proof of Theorem \ref{main theorem} the T-ideals of the minimal varieties
form a free graded semigroup freely generated by the T-prime ideals with grading defined by
$\deg T({\mathfrak R})=\text{\rm GKdim}(F_d({\mathfrak R}))$, see, e.g. \cite[Corollary 3.14]{D3}.
Since the greatest common divisor of the indices $(d-1)m^2+1$ of the coefficients of the power series
$a^{(d\text{\rm -f.g. generated})}(t)$ and $a^{(d\text{\rm -generated})}(t)$ is equal to 1,
it is sufficient to show that the radius of convergence of these power series is equal to 1.
This is obvious for the generating functions for $\text{\rm GKdim}(F_d(\text{\rm var}(M_m(K))))$
and for $\text{\rm GKdim}(F_d(\text{\rm var}(M_m(E))))$ because its coefficients are equal to 0 or 1. Since
\[
\sum_{1\leq a\leq b}t^{(d-1)(a^2+b^2)+2}\preceq \sum_{a\geq 1}t^{(d-1)a^2+1}\sum_{b\geq 1}t^{b^2+1}
\]
and the radius of convergence of the factors $\sum_{a\geq 1}t^{(d-1)a^2+1}$ and $\sum_{b\geq 1}t^{b^2+1}$ is equal to 1,
the same holds for the generating function for $\text{\rm GKdim}(F_d(\text{\rm var}(M_{a,b})))$.
\end{proof}

\begin{remark}
Repeating the arguments for the proof of the transcendence of the series $b^{(\text{\rm f.g. codim})}(t)$
in Theorem \ref{transcendental series for exponents} we can prove the transcendence of the series
$b^{(d\text{\rm -f.g. generated})}(t)$. Unfortunately, we cannot find arguments for the transcendence of $b^{(d\text{\rm -generated})}(t)$
although we believe that this is the case.
Similarly, we do not know whether the limits in Theorem \ref{asymptotics for GKdim} are
transcendental.
\end{remark}

\section{Free graded semigroups and free associative algebras}\label{Section Free semigroups}

Let $Y$ be an arbitrary set and let $\langle Y\rangle$ be the free unitary semigroup freely generated by the set $Y$.
As usually we say that $\langle Y\rangle$ is graded if there is a map $\deg:\langle Y\rangle\to{\mathbb N}=\{0,1,2,\ldots\}$
such that
\[
\deg(uv)=\deg(u)+\deg(v),\quad u,v\in\langle Y\rangle.
\]
In the sequel we assume that $\deg(y)>0$ for all $y\in Y$ and the set of the generators of degree $k$
\[
Y_k=\{y\in Y\mid \deg(y)=k\},\quad k=1,2,\ldots,
\]
is finite for any $k=1,2,\ldots$. Then the set $\langle Y\rangle^{(n)}$ of elements of degree $n$ in $\langle Y\rangle$ is also finite.
We fix the notation $\vert Y_k\vert=d_k$ and $\vert \langle Y\rangle^{(n)}\vert=g_n$, and introduce the generating functions
\[
d(t)=\sum_{k\geq 1}d_kt^k,\quad g(t)=\sum_{n\geq 0}g_nt^n.
\]

The following lemma is well known.

\begin{lemma}\label{1-1 correspondence with A-compositions}
Let $\langle Y\rangle$ be a free unitary graded semigroup, let $a(t)$ be the generating function which counts
the number of free generators of a given degree, and let $A$ be the associated to $a(t)$ multiset.
Then there is a canonical bijective correspondence between the elements of degree $n$ in $\langle Y\rangle$
and the $A$-compositions of $n$.
\end{lemma}

\begin{proof}
If $y_{ki}$, $i=1,\ldots,d_k$, are the free generators of degree $k$ in $Y$, and $k^{(i)}$, $i=1,\ldots,d_k$, are the $d_k$ copies of $k$ in $A$,
then we define a bijective map sending the element $u=y_{k_1i_i}\cdots y_{k_pi_p}\in \langle Y\rangle$ to the $A$-composition
$(k_1^{(i_i)},\ldots,k_p^{(i_p)})$.
\end{proof}

Hence we obtain immediately analogues of Theorem \ref{existence of exponent}, Corollaries \ref{transcendental radius of convergence}
and \ref{transcendental radius of convergence for A-compositions}, and Example \ref{Liouville number}.
We shall restate only Example \ref{Liouville number}.

\begin{example}\label{Liouville number for semigroups}
If the set $Y=\{y_{1!},y_{2!},\ldots\}$ consists of elements of degree $k!$, $k=1,2,\ldots$,
then for the sequence of the numbers $g_n$ of elements of degree $n$ in the free graded semigroup $\langle Y\rangle$ the limit
$\displaystyle \gamma=\lim_{n\to\infty}\sqrt[n]{g_n}$
exists and is a transcendental number.
\end{example}

Let $K$ be an arbitrary field. If $Y$ is a graded set, then the free associative algebra $K\langle Y\rangle$
is a $K$-vector space with the free semigroup $\langle Y\rangle$ as a basis.
Then the Hilbert series of $K\langle Y\rangle$ is the formal power series
\[
H(K\langle Y\rangle,t)=\sum_{n\geq 0}\dim(K\langle Y\rangle^{(n)})t^n,
\]
where $K\langle Y\rangle^{(n)}$ is the homogeneous component of degree $n$ of $K\langle Y\rangle$.
Clearly, $H(K\langle Y\rangle,t)$ is equal to the generating function $g(t)$
of the sequence $g_n=\vert \langle Y\rangle^{(n)}\vert$, $n=0,1,2,\ldots$.
Hence we can restate Theorem \ref{existence of exponent}, Corollaries \ref{transcendental radius of convergence}
and \ref{transcendental radius of convergence for A-compositions}, and Example \ref{Liouville number}
in the language of the coefficients of the Hilbert series of $K\langle Y\rangle$.

\section*{Acknowledgements}

An essential part of this project was carried out
during the Workshop ``Polynomial Identities in Algebras'',
September 16-20, 2019, organized by
Istituto Nazionale di Alta Matematica (INdAM), Roma, Italy,
and dedicated to the anniversary of Antonio Giambruno.
The author is very grateful to INdAM for the warm hospitality and the creative atmosphere during the Workshop.
Very special thanks are due to Olga Finogenova and Mikhail Zaicev. Without the fruitful discussions with them
and their valuable comments and suggestions during the Workshop this project would never start.


\begin{thebibliography}{99}

\bibitem{B1}
A. Berele,
Homogeneous polynomial identities,
Israel J. Math. {\bf 42} (1982), 258-272.

\bibitem{B2}
A. Berele,
Generic verbally prime PI-algebras and their GK-dimensions,
Commun. Algebra {\bf 21} (1993), No. 5, 1487-1504.

\bibitem{BR}
A. Berele, A. Regev,
On the codimensions of the verbally prime P.I. algebras,
Israel J. Math. {\bf 91} (1995), 239-247.

\bibitem{BeL}
G.M. Bergman, J. Lewin,
The semigroup of ideals of a fir is (usually) free,
J. Lond. Math. Soc., II. Ser. {\bf 11} (1975), 21-31.

\bibitem{C}
H. Cohn,
Note on almost-algebraic numbers,
Bull. Am. Math. Soc. {\bf 52} (1946), 1042-1045.

\bibitem{D1}
V. S. Drensky,
Extremal varieties of algebras. I,
Serdica {\bf 13} (1987), No. 4, 320-332.

\bibitem{D2}
V. S. Drensky,
Extremal varieties of algebras. II,
Serdica {\bf 14} (1988), No. 1, 20-27.

\bibitem{D3}
V. Drensky,
Gelfand-Kirillov dimension of PI-algebras,
in ``Methods in Ring Theory, Proc. of the Trento Conf.'',
Lect. Notes in Pure and Appl. Math. {\bf 198}, Dekker, 97-113, 1998.

\bibitem{D4}
V. Drensky,
Free Algebras and PI-Algebras,
Springer-Verlag, Singapore, 2000.

\bibitem{GZ1}
A. Giambruno, M. Zaicev,
On codimension growth of finitely generated associative algebras,
Adv. Math. {\bf 140} (1998), No. 2, 145-155.

\bibitem{GZ2}
A. Giambruno, M. Zaicev,
Exponential codimension growth of PI algebras: an exact estimate,
Adv. Math. {\bf 142} (1999), No. 2, 221-243.

\bibitem{GZ3}
A. Giambruno, M. Zaicev,
Minimal varieties of algebras of exponential growth,
Adv. Math. {\bf 174} (2003), No. 2, 310-323.

\bibitem{GZ4}
A. Giambruno, M. Zaicev,
Codimension growth and minimal superalgebras,
Trans. Am. Math. Soc. {\bf 355} (2003), No. 12, 5091-5117.

\bibitem{GZ5}
A. Giambruno, M. Zaicev,
Polynomial Identities and Asymptotic Methods,
Mathematical Surveys and Monographs {\bf 122}, AMS,
Providence, RI, 2005.

\bibitem{HeM}
S. Heubach, T. Mansour,
Combinatorics of Compositions and Words,
Discrete Mathematics and its Applications, CRC Press, Boca Raton, FL, 2009.

\bibitem{Ka}
H. Kaneko,
Algebraic independence of the values of power series with unbounded coefficients,
Ark. Mat. {\bf 55} (2017), No. 1, 61-87.

\bibitem{K1}
A.R. Kemer,
Varieties and ${\mathbb Z}_2$-graded algebras (Russian),
Izv. Akad. Nauk SSSR, Ser. Mat. {\bf 48} (1984), 1042-1059.
Translation: Math. USSR, Izv. {\bf 25} (1985), 359-374.

\bibitem{K3}
A.R. Kemer,
Representability of reduced-free algebras (Russian),
Algebra i Logika {\bf 27} (1988), 274-294.
Translation: Algebra and Logic {\bf 27} (1988), 167-184.

\bibitem{K2}
A.R. Kemer,
Ideals of Identities of Associative Algebras,
Translations of Math. Monographs {\bf 87}, AMS,
Providence, RI, 1991.

\bibitem{KL}
G.R. Krause, T.H. Lenagan,
Growth of Algebras and Gelfand-Kirillov Dimension, Revised ed.,
Graduate Studies in Mathematics  {\bf 22}, AMS,
Providence, RI, 2000.

\bibitem{L}
J. Liouville,
Sur des classes tr\`es \'etendues de quantit\'es dont valeur n'est ni alg\'ebrique,
ni m\^eme r\'educible \`a des irrationelles alg\'ebriques,
C.R. Acad. Sci., Paris, S\'er. A {\bf 18} (1844), 883-885.
J. Math. Pures Appl. {\bf 16} (1851), 133-142.

\bibitem{Ma1}
K. Mahler,
Arithmetic properties of lacunary power series with integral coefficients,
J. Aust. Math. Soc. {\bf 5} (1965), 56-64.

\bibitem{Ma2}
K. Mahler,
Lectures on Transcendental Numbers,
Edited and completed by B. Divi\v{s} and W.J. Le Veque,
Lecture Notes in Mathematics {\bf 546}, Springer-Verlag, Berlin-Heidelberg-New York, 1976.

\bibitem{M}
V.T. Markov,
The Gelfand-Kirillov dimension:
nilpotency, representability, non-matrix varieties (Russian),
Siberian School on Varieties of Algebraic Systems,
Abstracts, Barnaul, 1988, 43-45. Zbl. 685.00002.

\bibitem{MW}
L. Moser, E.L. Whitney,
Weighted compositions,
Can. Math. Bull. {\bf 4} (1961), 39-43.

\bibitem{P}
C. Procesi,
Non-commutative affine rings,
Atti Accad. Naz. Lincei, Mem., Cl. Sci. Fis. Mat. Nat., Sez. I, VIII Ser. {\bf 8} (1967), 239-255.

\bibitem{Re1}
A. Regev,
Existence of identities in $A \otimes B$,
Israel J. Math. {\bf 11} (1972), 131-152.

\bibitem{Re2}
A. Regev,
Codimensions and trace codimensions of matrices are asymptotically equal,
Israel J.Math. {\bf 47} (1984), 246-250.

\bibitem{RG-S}
J.C. Rosales, P.A. Garc\'{\i}a-S\'anchez,
Numerical Semigroups,
Developments in Mathematics {\bf 20}, Springer-Verlag, Dordrecht, 2009.

\end{thebibliography}
\end{document}